\documentclass[11pt,reqno]{amsart}
\usepackage{amsthm,amssymb,amsmath}
\usepackage{textcomp}

\usepackage[T1]{fontenc}

\newtheorem{theorem}{Theorem}
\newtheorem*{theorem*}{Theorem}
\newtheorem{lemma}{Lemma}

\newtheorem{proposition}{Proposition}

\theoremstyle{definition}

\newtheorem{definition}{\sc Definition}
\newtheorem*{definition*}{\bf Definition}

\newtheorem{example}{\bf Example}

\newtheorem{remark}{Remark}
\newtheorem*{remark*}{Remark}

\newtheorem*{example*}{\bf Example}

\newcommand{\loc}{{\rm loc}}

\begin{document}

\title{Sharp solvability for singular SDEs}

\author{Damir Kinzebulatov} 

\address{D\'{e}partement de math\'{e}matiques et de statistique, Universit\'{e} Laval, Qu\'{e}bec, QC, G1V 0A6, Canada}

\email{damir.kinzebulatov@mat.ulaval.ca}

\author{Yuliy A.\,Sem\"{e}nov}

\address{University of Toronto, Department of Mathematics, Toronto, ON, M5S 2E4, Canada}

\email{semenov.yu.a@gmail.com}

\thanks{The research of D.K.\,is supported by the NSERC (grant RGPIN-2017-05567)}



\begin{abstract}
The attracting inverse-square drift provides a prototypical counterexample to solvability of singular SDEs: 
if the coefficient of the drift is larger than a certain critical value, then no weak solution exists. 
We prove a positive result on solvability of singular SDEs where this critical value is attained from below (up to strict inequality) for the entire class of form-bounded drifts. This class contains e.g.\,the inverse-square drift, the critical Ladyzhenskaya-Prodi-Serrin class. The proof is based on a $L^p$ variant of De Giorgi's method.
\end{abstract}

\maketitle

\section{Introduction and main result}

The paper addresses the question: what are the minimal assumptions on a locally unbounded vector field $b:[0,\infty[ \times \mathbb R^d \rightarrow \mathbb R^d$, $d \geq 3$, also called a drift, such that the stochastic differential equation (SDE)
\begin{equation}
\label{sde}
dX_t=-b(t,X_t)dt + \sqrt{2}dB_t, \quad X_0=x \in \mathbb R^d
\end{equation}
admits a martingale (or weak) solution? Here $B_t$ is a standard Brownian motion in $\mathbb R^d$.
There is an extensive literature devoted to the search for such minimal assumptions, as well as to the question of what additional hypothesis on $b$ is required in order to ensure the uniqueness of the solution. The interest is motivated, in particular, by physical applications and applications to the theory of stochastic optimal control.

It is known that if $b$ is in the Ladyzhenskaya-Prodi-Serrin class
\begin{equation}
\label{LPS}
\tag{LPS}
|b| \in L_{\loc}^q([0,\infty[,L^r+L^\infty), \quad \frac{d}{r}+\frac{2}{q} < 1, \quad 2<q \leq \infty
\end{equation}
then equation \eqref{sde} has a unique in law martingale solution, see Portenko \cite{P}; the strong existence and uniqueness is due to Krylov-R\"{o}ckner \cite{KR}. In 2014, Beck-Flandoli-Gubinelli-Maurelli \cite{BFGM}, extending a method in \cite{CE}, gave the following counterexample to weak solvability of \eqref{sde} (among many results on the strong existence and uniqueness for \eqref{sde} with $b$ in the critical Ladyzhenskaya-Prodi-Serrin class
\begin{equation}
\label{LPS_c}
\tag{${\rm LPS}_c$}
|b| \in L_{\loc}^q([0,\infty[,L^r+L^\infty), \quad \frac{d}{r}+\frac{2}{q} \leq 1, \quad 2 < q \leq \infty).
\end{equation}

\begin{example}
\label{ex}
Consider the inverse-square drift $b(x)=\sqrt{\delta} \frac{d-2}{2}|x|^{-2}x$ ($d \geq 3$). 

The following is true:

(a) If $\delta>4(\frac{d}{d-2})^2$, then equation \eqref{sde} with the initial point $x=0$ has no weak solution. 

(b) If $\delta>4$, then for every $x \neq 0$ the solution to \eqref{sde} arrives at the origin  in finite time with positive probability.
\end{example}

The vector field in Example \ref{ex} has a stronger singularity than any vector field in (${\rm LPS}_c$).
Intuitively, when $\delta>4$, the attraction to the origin is so strong that the process, even  starting at $x \neq 0$, does not look like a Brownian motion. See also  \cite{W} regarding (b).

In fact, it has been known for a long time that the value $\delta=4$ is critical, although in a different context: the theory of operator $-\Delta + b \cdot \nabla$. 
More precisely, let $b$ be a form-bounded vector field, i.e.
$$
|b|^2 \leq \delta (-\Delta) + g_\delta \quad \text{ in the sense of quadratic forms (see below)}
$$
(e.g.\,$b$ in Example \ref{ex} is form-bounded). It was proved in \cite{KS, S} that if the form-bound $\delta$ satisfies $\delta<4$, then there exists a quasi contraction strongly continuous Markov evolution family in $L^p$, $p > \frac{2}{2-\sqrt{\delta}}$ that delivers a unique weak solution to Cauchy problem 
\begin{equation}
\label{par}
(\partial_t-\Delta + b \cdot \nabla)u=0, \quad u(0)=f \in L^p
\end{equation}
  (a strong solution if $b=b(x)$). Here one expects, of course, in the time-homogeneous case, $$u(t,x)=\mathbf E[f(X_t)].$$ The interval of contraction solvability can be closed to $[\frac{2}{2-\sqrt{\delta}},\infty[$ and is sharp, see \cite{KiS}. 
Now, as $\delta \uparrow 4$, this interval disappears, and with it the theory of the operator $-\Delta + b \cdot \nabla$.

\begin{definition}
\label{def1}
A vector field $b:[0,\infty[ \times \mathbb R^d \rightarrow \mathbb R^d$ is said to be form-bounded  if $|b| \in L^2_{\loc}([0,\infty[ \times \mathbb R^d)$ and there exist a constant $\delta>0$ and a function $0 \leq g_\delta \in L^1_{\loc}([0,\infty[)$ such that
\begin{align*}
 \int_0^\infty \|b(t,\cdot)\xi(t,\cdot)\|_2^2 dt   \leq \delta \int_0^\infty\|\nabla \xi(t,\cdot)\|_2^2 dt+\int_0^\infty g_\delta(t)\|\xi(t,\cdot)\|_2^2dt
\end{align*}
for all $\xi \in C^\infty_c([0,\infty[ \times \mathbb R^d)$ (written as $b \in \mathbf{F}_\delta$). 
 
Here  $\|\cdot\|_p:=\|\cdot\|_{L^p(\mathbb R^d)}$.
\end{definition}

Examples of form-bounded vector fields include: 
$b \in ({\rm LPS}_c)$, 
the vector field in Example \ref{ex} or, more generally, vector fields $b=b(x)$ with $|b|$ in the weak $L^d$ class or the Campanato-Morrey class. One can construct, for every $\varepsilon>0$, a form-bounded $b=b(x)$ with $|b| \not\in L^{2+\varepsilon}$. See \cite{KiS,KM} for details and other examples.

The present paper deals with the SDE \eqref{sde} with a form-bounded vector field $b \in \mathbf{F}_\delta$. In \cite{KM} it was proved that if the form-bound $\delta$ satisfies $\delta<d^{-2}$, then, for every $x \in \mathbb R^d$, \eqref{sde} has a weak solution that is unique in a large class, and is given by a Feller evolution family. In the time-homogeneous case $b=b(x)$ the result is stronger, that is, $b$ is required to be form-bounded with $\delta<\big(2(d-2)^{-1} \wedge 1\big)^2$ \cite{KiS2}, or even only weakly form-bounded \cite{KiS3}, which allows to treat vector fields $b$ that are a priori only in $L^1_{\loc}$. In both cases the solutions are determined by a Feller semigroup, and are unique among weak solutions that are constructed using an approximation of $b$ by smooth vector fields that do not increase the form-bound $\delta$ of $b$. 

See also Krylov \cite{Kr} regarding Markov weak solvability of \eqref{sde} for $|b| \in L_{\loc}^q([0,\infty[,L^r+L^\infty)$, $\frac{d}{r}+\frac{1}{q} \leq 1$, $r \geq d$, $q \geq 1$ (in \cite{Kr} the SDE can in fact have just measurable diffusion coefficients).

In \cite{KiS3, KiS2,KM}, the construction of the Feller evolution family (semigroup) and the weak solution to \eqref{sde} is based on quite strong gradient estimates on solution $u$ to the parabolic  equation \eqref{par} (solution $v$ to the elliptic equation  $(\lambda-\Delta + b\cdot \nabla) v=f$) with $b \in \mathbf{F}_\delta$. In \cite{KM} and \cite{KiS2} these estimates are, respectively,
$$\|\nabla u\|_{L^\infty([0,T],L^q)},\|\nabla |\nabla u|^{\frac{q}{2}}\|^{\frac{2}{q}}_{L^2([0,T],L^2)} \leq C \|\nabla f\|_q \quad \text{ for } q \in ]d,k(\delta)[ \quad \text{ if } \delta<d^{-2}
$$
and
$$
\|\nabla v\|_{L^q},  \|\nabla |\nabla v|^{\frac{q}{2}}\|^{\frac{2}{q}}_{L^2} \leq C' \|f\|_q \quad \text{ for } q \in ]2 \vee (d-2),k'(\delta)[ \quad \text{ if } \delta<\big(2(d-2)^{-1} \wedge 1\big)^2.
$$ 
Extending these estimates to vector fields $b$ whose form-bound $\delta$ surpasses $\big(2(d-2)^{-1} \wedge 1\big)^2$ (not to mention $\delta$ going up to $4$) is  problematic if not impossible. Thus, there is a gap between the hypothesis on $\delta$ in \cite{KiS3,KiS2,KM} and in Example \ref{ex}. 
The purpose of this paper is to fill this gap.

\begin{theorem}
\label{thm1}
Let $d \geq 3$, $b \in \mathbf{F}_\delta$. If $\delta<4$, then for every $x \in \mathbb R^d$ there exists a martingale solution to \eqref{sde}.
\end{theorem}

Theorem \ref{thm1} shows that Example \ref{ex} is essentially sharp, at least as $d \rightarrow \infty$.
A crucial feature of Theorem \ref{thm1} is that it attains $\delta=4$ (up to strict inequality) for the entire class of form-bounded vector fields.

We leave aside the important issues of the Feller property/uniqueness of the constructed martingale solution. Let us only mention that if one is willing to impose additional assumptions on ${\rm div}\,b$ (namely, the Kato class condition), then Nash's method allows to obtain two-sided Gaussian bounds on the fundamental solution to \eqref{par}, see \cite{KiS4}, from which the Feller property follows.

We prove the main analytic result (Proposition \ref{mp} below) using De Giorgi's iterations. They are carried out in $L^p$, $p>\frac{2}{2-\sqrt{\delta}}$, $p \geq 2$ rather than in the standard for the De Giorgi method $L^2$ space, as is needed to handle $1 \leq \delta<4$. In this regard, let us make a trivial observation that passing to $L^p$ right away, using the fact that $u^\frac{p}{2}$ is a subsolution and then applying to $u^\frac{p}{2}$ the standard  De Giorgi iteration procedure in $L^2$, does not allow to treat $1 \leq \delta<4$. We will have to 
follow the iteration procedure from the very beginning and adjust it accordingly.

Earlier, De Giorgi's method in $L^2$ was used by Zhang-Zhao \cite{ZZ}, Zhao \cite{Zh}, R\"{o}ckner-Zhao \cite{RZ}. They prove, in particular, results on weak well-posedness of \eqref{sde} with a zero-divergence $b$ satisfying $$|b| \in L_{\loc}^q([0,\infty[,L^r+L^\infty), \quad \frac{d}{r}+\frac{2}{q} < 2.$$ Similarly to these papers, we apply a tightness argument to construct a martingale solution once Proposition \ref{mp} is established, see Section \ref{sde_subsect}. We also refer to Hara \cite{Ha} for the proof of H\"{o}lder continuity  of solutions to elliptic equations with $b \in \mathbf{F}_\delta$, $\delta<1$ using Moser's method in $L^2$.

Finally, we note that passing to the $L^p$ variant of De Giorgi's method does not exclude other singular drift perturbations known to be  amenable in $L^2$ (cf.\,\cite{Ha,Z}). For instance, the assertion of Theorem \ref{thm1} is also valid for $b=b_1+b_2$, where $b_1 \in \mathbf{F}_{\delta_1}$, $\delta_1<4$ and $b_2$ satisfies:

1) there exists $0<a \leq 1$ such that
$|b_2| \in L^{1+a}_{\loc}([0,\infty[ \times \mathbb R^d)$ and 
\begin{align}
\label{fb2}
 \int_0^\infty \langle |b_2(t,\cdot)|^{1+a}\xi^2(t,\cdot)\rangle dt   \leq \delta_2 \int_0^\infty\|\nabla \xi(t,\cdot)\|_2^2 dt+\int_0^\infty g_\delta(t)\|\xi(t,\cdot)\|_2^2dt
\end{align}
for all $\xi \in C^\infty_c([0,\infty[ \times \mathbb R^d)$, for some $0<\delta_2<\infty$ and $0 \leq g_{\delta_2} \in L^1_{\loc}([0,\infty[)$.  Here and everywhere below,
$$
\langle f,g\rangle = \langle f g\rangle :=\int_{\mathbb R^d}f gdx.
$$

2) the divergence $({\rm div}\,b_2)_+ \in L^1_{\loc}([0,\infty[ \times \mathbb R^d)$ and
\begin{equation}
\label{fb3}
\int_0^\infty \big\langle ({\rm div}\,b_2)_+(t,\cdot)\xi^2(t,\cdot)\big\rangle dt   \leq \nu \int_0^\infty\|\nabla \xi(t,\cdot)\|_2^2 dt+\int_0^\infty g_{\nu}(t)\|\xi(t,\cdot)\|_2^2dt
\end{equation}
with $\nu<4-2\sqrt{\delta_1}$ for some $0 \leq g_{\nu} \in L^1_{\loc}([0,\infty[)$. We say that $({\rm div}\,b_2)_+$ is a form-bounded potential (for instance, it can be a function in the weak $L^{\frac{d}{2}}$ space).
For details, see Remark \ref{div_rem} below. There we explain that $\nu<4-2\sqrt{\delta_1}$ suffices, provided that we prove the energy inequality in $L^p$ for $p$ such that $\nu<4\frac{p-1}{p}-2\sqrt{\delta_1}$. If we stay in $L^2$, then we have to impose more restrictive condition $\nu<2-2\sqrt{\delta_1}$. 

The class \eqref{fb2} is essentially twice more singular than $\mathbf{F}_\delta$.
It first appeared in Q.\,S.\,Zhang \cite{Z}, where the author used Moser's method in $L^2$ to prove, assuming that the vector field has zero divergence and satisfies \eqref{fb2}, the local boundedness of weak solutions to the corresponding parabolic equation, and applied this result to study equations of Navier-Stokes in $\mathbb R^3$.

\bigskip

\section{Proof of Theorem \ref{thm1}}

\subsection{De Giorgi's iterations in $L^p$}

In the next two propositions, $u$ is the solution to Cauchy's problem for inhomogeneous Kolmogorov equation
\begin{equation}
\label{k_eq3}
(\partial_t  - \Delta  + b \cdot \nabla) u=|\mathsf{h}|f, \quad u(0)=0.
\end{equation} 
where
$$b \in \mathbf{F}_\delta \cap C^\infty_c(]0,\infty[ \times \mathbb R^d), \quad \delta<4,$$
$$\mathsf{h} \in \mathbf{F}_\nu \cap C^\infty_c(]0,\infty[ \times \mathbb R^d), \quad \nu<\infty \qquad \text{ and } f \in C_c.$$
Since the coefficients of \eqref{k_eq3} are smooth with compact support, the solution $u$ exists and is sufficiently regular to justify the manipulations with the equation below.

Set
$$
p_\delta:=\frac{2}{2-\sqrt{\delta}}.
$$

We will call a constant generic if it only depends on $d$, $p$, $\delta$, $\nu$ (and $T>0$, in case we work over a fixed finite time interval $[0,T]$).

\begin{proposition}[Energy inequality]
\label{c_prop}
Let $u$ be the solution to Cauchy problem \eqref{k_eq3}.
Let $p > p_\delta$, $p \geq 2$. Set $u_c:=(u-c)_+$, $c \in \mathbb R$. Fix $T>0$ and $\eta \in C_c^\infty(\mathbb R^d)$. Then,  $0 < t-s \leq T$,
\begin{align}
\sup_{\vartheta \in [s,t]}\langle u_c^p(\vartheta) \eta^2\rangle &   + \int_s^t \langle |\nabla (\eta u_c^{\frac{p}{2}})|^2 \rangle\label{c_ineq}  \\
& \leq C_1\langle u_c^p(s) \eta^2\rangle  + C_2\int_s^t \langle u_c^p|\nabla\eta|^2\rangle + C_3  \int_s^t \big\langle \bigl(\mathbf{1}_{\{|\mathsf{h}| \geq 1\}} + \mathbf{1}_{\{|\mathsf{h}| < 1\}}|\mathsf{h}|^p\bigr)\mathbf{1}_{\{u>c\}}|f|^p \eta^2 \big\rangle \notag
\end{align}
for generic constants $C_1$-$C_3>0$.
\end{proposition}

The last term in the RHS has this form because this is what will be need in the next section. There we will consider 1) $\mathsf{h}=b$, in order to apply a tightness argument to construct a candidate for the martingale solution to \eqref{sde}, and 2) $\mathsf{h}=b_{m_1}-b_{m_2}$ where $\{b_m\} \subset \mathbf{F}_\delta \cap C^\infty_c(]0,\infty[ \times \mathbb R^d)$ is an approximation of a (discontinuous) $b \in \mathbf{F}_\delta$, in order to pass to the limit in the martingale problem; we will take $f=|\nabla \varphi|$, where $\varphi \in C_c^2$ is a test function in the martingale problem.

\begin{proof}[Proof of Proposition \ref{c_prop}]
Put for brevity $v:=u_c$. It suffices to prove
\begin{align}
\sup_{\vartheta \in [s,t]}\langle v^p(\vartheta) \eta^2\rangle &  + \int_s^t \langle |\nabla v^{\frac{p}{2}}|^2 \eta^2 \rangle \label{req_ineq} \\
& \leq  C_1 \langle v^p(s) \eta^2\rangle  + C_2\int_s^t \langle v^p|\nabla\eta|^2\rangle + C_3  \int_s^t \big\langle \bigl(\mathbf{1}_{\{|\mathsf{h}| \geq 1\}} + \mathbf{1}_{\{|\mathsf{h}| < 1\}}|\mathsf{h}|^p\bigr)\mathbf{1}_{\{v>0\}}|f|^p \eta^2 \big\rangle . \notag
\end{align}
We multiply equation \eqref{k_eq3} by $v^{p-1} \eta^2$ and integrate to obtain
\begin{align}
\langle v^p(t) \eta^2\rangle  - \langle v^p(s) \eta^2\rangle & + \frac{4(p-1)}{p}\int_s^t \langle |\nabla v^{\frac{p}{2}}|^2\eta^2 \rangle \notag \\
& \leq 4\left|\int_s^t \langle \nabla v^{\frac{p}{2}}, v^{\frac{p}{2}}\eta \nabla \eta\rangle \right| + 2\left| \int_s^t \langle b\cdot \nabla v^{\frac{p}{2}}, v^{\frac{p}{2}}\eta^2\rangle \right| + \left|\int_s^t \langle |\mathsf{h}|f v^{p-1}\eta^2\rangle\right| \notag
\end{align}
where we estimate in the RHS:

1.~
$$
4\left|\int_s^t \langle \nabla v^{\frac{p}{2}}, v^{\frac{p}{2}}\eta \nabla \eta\rangle \right| \leq 2\varepsilon_1 \int_s^t \langle |\nabla v^{\frac{p}{2}}|^2\eta^2 \rangle + \frac{2}{\varepsilon_1} \int_s^t \langle v^p|\nabla \eta|^2\rangle \qquad  (\varepsilon_1>0).
$$

2.~
\begin{align*}
2\left| \int_s^t \langle b\cdot \nabla v^{\frac{p}{2}}, v^{\frac{p}{2}}\eta^2\rangle \right| & \leq \frac{1}{\sqrt{\delta}}\int_s^t \langle |b|^2 v^p\eta^2\rangle + \sqrt{\delta} \int_s^t \langle |\nabla v^{\frac{p}{2}}|^2\eta^2 \rangle \\
& (\text{we are using $b \in \mathbf{F}_\delta$}) \\
& \leq \frac{1}{\sqrt{\delta}}\biggl( \delta \int_s^t \langle |\nabla (\eta v^{\frac{p}{2}})|^2 \rangle + \int_s^t g_\delta \langle  v^p  \eta^2\rangle \biggr) + \sqrt{\delta} \int_s^t \langle |\nabla v^{\frac{p}{2}}|^2\eta^2 \rangle.
\end{align*}
We bound $|\nabla (\eta v^{\frac{p}{2}})|^2 \leq (1+\varepsilon_2)\langle |\nabla v^{\frac{p}{2}}|^2\eta^2 \rangle + (1+\varepsilon_2^{-1})\langle v^p |\nabla \eta|^2\rangle $, $\varepsilon_2>0$.
Then
\begin{align*}
\sup_{\vartheta \in [s,t]}\langle v^p(\vartheta) \eta^2\rangle & + \biggl[\frac{4(p-1)}{p} - (2+\varepsilon_2)\sqrt{\delta}-2\varepsilon_1 \biggr]\int_s^t \langle |\nabla v^{\frac{p}{2}}|^2\eta^2 \rangle \notag \\
& \leq  \langle v^p(s) \eta^2\rangle + \frac{1}{\sqrt{\delta}}\int_s^t g_\delta \langle v^p \eta^2 \rangle + C_2' \int_s^t \langle v^p|\nabla \eta|^2\rangle + \left|\int_s^t \langle |\mathsf{h}|f v^{p-1}\eta^2\rangle\right|.
\end{align*}
Now, assuming first that $T>0$ is sufficiently small so that $\frac{1}{\sqrt{\delta}}\bigl(\int_s^t g_\delta\bigr)<\frac{1}{3}$, we obtain
\begin{align}
\frac{2}{3}\sup_{\vartheta \in [s,t]}\langle v^p(\vartheta) \eta^2\rangle  & + \biggl[\frac{4(p-1)}{p} - (2+\varepsilon_2)\sqrt{\delta}-2\varepsilon_1 \biggr]\int_s^t \langle |\nabla v^{\frac{p}{2}}|^2\eta^2 \rangle \notag \\
& \leq C_1'\langle v^p(s) \eta^2\rangle + C_2' \int_s^t \langle v^p|\nabla \eta|^2\rangle +  C_3'\left|\int_s^t \langle |\mathsf{h}|f v^{p-1}\eta^2\rangle\right|. \label{v_ineq}
\end{align}
Next, using the reproduction property, we extend the last inequality to arbitrary $T>0$ (at expense of increasing $C_i'=C_i'(T)$, $i=1,2,3$).

It remains to estimate the last term in the RHS of \eqref{v_ineq}:
\begin{align}
&\left|\int_s^t \langle |\mathsf{h}|f v^{p-1}\eta^2\rangle \right| \notag \\
&\leq  \int_s^t \langle \mathbf{1}_{|\mathsf{h}| \geq 1}|\mathsf{h}|^{\frac{2}{p'}}|f| |v|^{p-1}\eta^2\rangle + \int_s^t \langle \mathbf{1}_{|\mathsf{h}| < 1}|\mathsf{h}||f| |v|^{p-1}\eta^2\rangle  =: I_1+I_2, \label{h_term}
\end{align}
where, by Young's inequality,
\begin{align*}
I_1 & \leq \frac{\varepsilon_3^{p'}}{p'} \int_s^t \langle \mathbf{1}_{|\mathsf{h}| \geq 1}|\mathsf{h}|^2 v^p\eta^2\rangle + \frac{\varepsilon_3^{-p}}{p}\int_s^t \langle \mathbf{1}_{|\mathsf{h}| \geq 1}\mathbf{1}_{\{v>0\}}|f|^p \eta^2\rangle \qquad (\varepsilon_3>0) \\
& (\text{we are using $\mathsf{h} \in \mathbf{F}_\nu$}) \\
& \leq \frac{\varepsilon_3^{p'}}{p'} \biggl( \nu \int_s^t |\nabla(\eta v^{\frac{p}{2}}) |^2\rangle + \int_s^t g_\nu \langle v^p\eta^2\rangle\biggr) + \frac{\varepsilon_3^{-p}}{p}\int_s^t \langle \mathbf{1}_{|\mathsf{h}| \geq 1}\mathbf{1}_{\{v>0\}}|f|^p \eta^2\rangle,
\end{align*}
and
$$
I_2 \leq  \frac{\varepsilon_4^{p'}}{p'} \int_s^t \langle \mathsf{1}_{|\mathsf{h}| <1} v^p\eta^2 \rangle + \frac{\varepsilon_4^{-p}}{p}\int_s^t \langle \mathbf{1}_{|\mathsf{h}| < 1}|\mathsf{h}|^p\mathbf{1}_{\{v>0\}}|f|^p \eta^2\rangle \qquad (\varepsilon_4>0).
$$
Inserting these estimates in \eqref{v_ineq} and taking care of the term $\int_s^t g_\nu \langle v^p\eta^2\rangle$ in the same way as we did 
above, we arrive at
\begin{align*}
& C\sup_{r \in [s,t]}\langle v^p(r) \eta^2\rangle   + \biggl[\frac{4(p-1)}{p} - (2+\varepsilon_2)\sqrt{\delta}-2\varepsilon_1 - \frac{\varepsilon_3^{p'}}{p'}\nu - \frac{\varepsilon_4^{p'}}{p'}  \biggr]\int_s^t \langle |\nabla v^{\frac{p}{2}}|^2\eta^2 \rangle \notag \\
& \leq C_1''\langle v^p(s) \eta^2\rangle + C_2'' \int_s^t \langle v^p|\nabla \eta|^2\rangle +  C_3''\int_s^t \big\langle \bigl(\mathbf{1}_{\{|\mathsf{h}| \geq 1\}} + \mathbf{1}_{\{|\mathsf{h}| < 1\}}|\mathsf{h}|^p\bigr)\mathbf{1}_{\{v>0\}}|f|^p \eta^2 \big\rangle,
\end{align*}
where the appropriate constant $C$ is positive provided that $\varepsilon_3$, $\varepsilon_4$ are sufficiently small.
Note that $\frac{4(p-1)}{p}-2\sqrt{\delta}>0$ if and only if $p>p_\delta$. Since the latter is a strict inequality, we can and will select  $\varepsilon_i$ ($i=1,2,3,4$) sufficiently small so that the coefficient of $\int_s^t \langle |\nabla v^{\frac{p}{2}}|^2\eta^2 \rangle$ is positive. We arrive at \eqref{req_ineq}, as needed.
\end{proof}

\begin{remark}Apart from the weight $\eta$ with compact support, we will also consider the weight
$$
\rho(x)=(1+\kappa|x|^{-2})^{-\beta}, \quad \beta>\frac{d}{4}, \quad \kappa>0.
$$
Then, in the assumptions of Proposition \ref{c_prop}, assuming that $\kappa$ is chosen sufficiently small, we  have for every $p>p_c$, $p \geq 2$, for all $0 \leq s \leq t$,
\begin{align}
\sup_{\vartheta \in [s,t]}\langle u_c^p(\vartheta)\rho^2 \rangle &  + \int_s^t \langle |\nabla (\rho  u_c^{\frac{p}{2}})|^2 \rangle \notag \\
& \leq  C_1\langle u_c^p(s) \rho^2\rangle  + C_2  \int_s^t \big\langle \bigl(\mathbf{1}_{\{|\mathsf{h}| \geq 1\}} + \mathbf{1}_{\{|\mathsf{h}| < 1\}}|\mathsf{h}|^p\bigr)\mathbf{1}_{\{u>c\}}|f|^p\rho^2\big\rangle. \label{c_ineq2} 
\end{align}
The proof essentially repeats the proof of \eqref{c_ineq} (we use $|\nabla \rho| \leq \beta\sqrt{\kappa} \rho$ at the last step to get rid of the $C_2$ term in \eqref{c_ineq}).
\end{remark}

\begin{lemma}[{\cite[Sect.7.2]{G}}]
\label{dg_lemma}
If $\{y_m\}_{m=0}^\infty \subset \mathbb R_+$ is a nondecreasing sequence such that
$$
y_{m+1} \leq N C_0^m y^{1+\alpha}_m
$$
for some $C_0>1$, $\alpha>0$, and
$$
y_0 \leq N^{-\frac{1}{\alpha}}C_0^{-\frac{1}{\alpha^2}}.
$$
Then
$$
\lim_m y_m=0.
$$
\end{lemma}

\begin{proposition} 
\label{mp}
Let $u$ be the solution to Cauchy problem \eqref{k_eq3}. Fix $T>0$ and $1<\theta<\frac{d}{d-1}$.
For all $p>p_\delta$, $p \geq 2$, there exists a generic constant $K$ such that 
\begin{align}
\sup_{[0,T] \times B(0,\frac{1}{2})}u_+  & \leq 2 \biggl(\int_0^{T}\big\langle \big(\mathbf{1}_{\{|\mathsf{h}| \geq 1\}} + \mathbf{1}_{\{|\mathsf{h}| < 1\}}|\mathsf{h}|^p\big)^{\theta'} |f|^{p\theta' }\mathbf{1}_{B(0,1)}\big\rangle\biggr)^{\frac{1}{p\theta'}} \label{mp_f}
  \\
	& + K \biggl( \int_0^{T} \langle u_+^p \mathbf{1}_{B(0,1)}\rangle + \biggl(\int_0^{T} \langle u_+^{ p\theta} \mathbf{1}_{B(0,1)}\rangle\biggr)^{\frac{1}{\theta}} \biggr)^{\frac{1}{p}}, \qquad \theta'=\frac{\theta}{\theta-1}. \notag
\end{align}

\end{proposition}

\begin{proof}[Proof of Proposition \ref{mp}]
Set $$R_m:=\frac{1}{2}(1+2^{-m}), \quad B_m:=B(0,R_m),$$
$$
M_m:=M(2-2^{-m}) 
$$
for a constant $M>0$ to be determined later.
Put $\eta_m:=\eta_{R_{m},R_{m-1}}$ where $\eta_{r,R}$ is a fixed family of smooth cutoff functions 
\begin{equation}
\label{eta_f}
\eta_{r,R} =1 \text{ in } B(0,r), \quad \eta_{r,R} = 0 \text{ in $\mathbb R^d - B(0,R)$}, \quad |\nabla \eta_{r,R}| \leq \frac{c_0}{4}(R-r)^{-1} \text{ for $0<r<R$}.
\end{equation}
Then $|\nabla \eta_m| \leq c_02^m$.
Define
$$
u_m:=(u-M_m)_+
$$
and
$$E_m:=\sup_{\vartheta \in [0,T]} \langle u_m^p(\vartheta)\eta_m^2\rangle + \int_0^{T} \langle |\nabla (\eta_m u_m^{\frac{p}{2}})|^2\rangle,
$$
$$
U_m:=\int_0^{T} \langle u_m^p \mathbf{1}_{B_m}\rangle + \biggl(\int_0^{T} \langle u_m^{p\theta} \mathbf{1}_{B_m}\rangle\biggr)^{\frac{1}{\theta}}.
$$

By Proposition \ref{c_prop}, using H\"{o}lder's inequality, we have for all $0 \leq t \leq T$
\begin{align}
\sup_{\vartheta \in [0,t]}\langle u_{m+1}^p(\vartheta)\eta_{m+1}^2\rangle & +  \int_0^t \langle |\nabla (\eta_{m+1}u_{m+1}^{\frac{p}{2}})|^2 \rangle \notag \\
& \leq C_2 c_0^24^m \int_0^t \langle u_{m+1}^p \mathbf{1}_{B_m} \rangle  + C_3 H^{\frac{1}{\theta'}} \big|\{u_{m+1}>0 \} \cap [0,t] \times B_{m}\big| ^{\frac{1}{\theta}}, \label{h_eq}
\end{align}
where
$$
H:= \int_0^{T}\big\langle \big(\mathbf{1}_{\{|\mathsf{h}| \geq 1\}} + \mathbf{1}_{\{|\mathsf{h}| < 1\}}|\mathsf{h}|^p\big)^{\theta'} |f|^{p\theta' }\mathbf{1}_{B(0,1)}\big\rangle.
$$
We estimate the last term in \eqref{h_eq}:
\begin{align}
\big|\{u_{m+1}>0 \} \cap [0,t] \times B_{m}\big| ^{\frac{1}{\theta}} & =
\biggl(\int_{0}^t \langle \mathbf{1}_{\{u_{m}>M2^{-m-1}\}}\mathbf{1}_{B_m}\rangle \biggr)^{\frac{1}{\theta}} \notag \\
& \leq (M2^{-m-1})^{-p}\biggl(\int_{0}^t \langle u^{p\theta}_m\mathbf{1}_{\{u_{m}>M2^{-m-1}\}}\mathbf{1}_{B_m}\rangle \biggr)^{\frac{1}{\theta}} \label{lt}.
\end{align}
We assume from now on that $M$ satisfies $M^p \geq H^{\frac{1}{\theta'}}$.
Then \eqref{h_eq} and \eqref{lt} yield
\begin{equation}
\label{e1}
E_{m+1} \leq C_4^m U_m \quad \text{ for appropriate constant $C_4$}.
\end{equation}

Next, using the Sobolev Embedding Theorem, we have
\begin{align*}
\sup_{[0,T]} \langle u_{m+1}^p \mathbf{1}_{B_{m+1}}\rangle & + c_S \int_0^{T} \|\mathbf{1}_{B_{m+1}}u_{m+1}\|^p_{\frac{pd}{d-2}} \leq E_{m+1}
\end{align*}
Applying H\"{o}lder's inequality and Young's inequality, we have
\begin{align*}
c\|\mathbf{1}_{B_{m+1}}u_{m+1}\|^p_{L^{2p}([0,T],L^{\frac{pd}{d-1}})} \leq \sup_{[0,T]} \langle u_{m+1}^p \mathbf{1}_{B_{m+1}}\rangle  + c_S \int_0^{T} \|\mathbf{1}_{B_{m+1}}u_{m+1}\|^p_{\frac{pd}{d-2}}
\end{align*}
for a $c>0$.
Next, applying H\"{o}lder's inequality to both terms in the definition of $U_{m+1}$, we obtain, for appropriate $\alpha>0$,
\begin{align*}
U_{m+1}  \leq c_2\|\mathbf{1}_{B_{m+1}}u_{m+1}\|^p_{L^{2p}([0,T],L^{\frac{pd}{d-1}})}\big|\{u_{m+1}>0\} \cap [0,T] \times  B_m\big|^\frac{\alpha}{\theta}.
\end{align*}
Combining this with the previous estimate, we have
\begin{equation}
\label{e2}
U_{m+1} \leq c_2c^{-1}E_{m+1}\big|\{u_{m+1}>0\} \cap [0,T] \times B_m\big|^\frac{\alpha}{\theta}
\end{equation}

Now, \eqref{e1} and \eqref{e2} yield
\begin{align*}
U_{m+1} & \leq c_2c^{-1}C_4^m U_m \big|\{u_{m+1}>0\} \cap [0,T] \times  B_m\big|^\frac{\alpha}{\theta}.
\end{align*}
Applying \eqref{lt} to the last multiple, we obtain
$$
U_{m+1} \leq M^{-p\alpha}C_5^m U_m^{1+\alpha}
$$
for constant $C_5=C_5(C_4,c,c_2,\alpha)$.

To end the proof, we fix $M$ by
$
M = H^{\frac{1}{p\theta'}} + C_5^{\frac{1}{p\alpha^2}}U_0^{\frac{1}{p}}.
$
Then $U_0 \leq C_5^{-\frac{1}{\alpha^2}}M^p$. We now apply Lemma \ref{dg_lemma} (with $N=M^{-p\alpha}$) to obtain
$$
\lim_m U_m=0.
$$
On the other hand, 
$$
\int_0^{T}(u-2M)_+^p \mathbf{1}_{B(0,\frac{1}{2})} \leq \lim_m U_m.
$$
It follows that
\begin{align*}
\sup_{[0,T] \times B(0,\frac{1}{2})}u_+  \leq 2M & \leq 2 \biggl(\int_0^{T}\big\langle \big(\mathbf{1}_{\{|\mathsf{h}| \geq 1\}} + \mathbf{1}_{\{|\mathsf{h}| < 1\}}|\mathsf{h}|^p\big)^{\theta'} |f|^{p\theta' }\mathbf{1}_{B(0,1)}\big\rangle\biggr)^{\frac{1}{p\theta'}} \\
&  + K \biggl( \int_0^{T} \langle u_+^p \mathbf{1}_{B(0,1)}\rangle + \biggl(\int_0^{T} \langle u_+^{ p\theta} \mathbf{1}_{B(0,1)}\rangle\biggr)^{\frac{1}{\theta}} \biggr)^{\frac{1}{p}}
\end{align*}
for a generic constant $K$, as claimed.
\end{proof}

\begin{remark}
\label{div_rem}
Let $\eta \in C_c^\infty$ be a refined cutoff function satisfying $|\nabla \eta| \leq c\eta^{1-\gamma}$ for some $0<\gamma<1$, $c>0$. In fact, 
the weights $\eta=\eta_{r,R}$ in \eqref{eta_f} can be chosen to satisfy this bound with generic $\gamma$, $c_0$:
$$
|\nabla \eta_{r,R}| \leq c_0 (R-r)^{-1} \eta_{r,R}^{1-\gamma}, \quad 0<r<R.
$$
See \cite{BS,Z}.
With such choice of the weights, Proposition \ref{c_prop} and thus Proposition \ref{mp} are also valid for $b=b_1+b_2$ where $b_1 \in \mathbf{F}_{\delta_1}$, $\delta_1<4$, and $b_2$ satisfies \eqref{fb2}, \eqref{fb3} with  $\nu<\frac{4(p-1)}{p}-2\sqrt{\delta_1}$, $p>\frac{2}{2-\sqrt{\delta_1}}$,
and $\mathsf{h}$ satisfies \eqref{fb2} with some form-bound $\nu<\infty$.
Indeed, we only need to complement the proof of Proposition \ref{c_prop} by evaluating, using the integration by parts,
\begin{align*}
-2 \int_s^t \langle b_2\cdot \nabla v^{\frac{p}{2}}, v^{\frac{p}{2}}\eta^2\rangle  = \int_s^t \langle {\rm div\,}b_2, v^p\eta^2\rangle  + 2\int_s^t \langle b_2 v^p \eta \nabla \eta\rangle
\end{align*}
and then estimating the RHS from above as follows.
We apply the form-boundedness condition on $({\rm div\,}b_2)_+$, i.e.\,\eqref{fb3}. As for the last term, we have
for every $\varepsilon_5>0$, by Young's inequality,
\begin{align*}
2\int_s^t \langle b_2 v^p, \eta\nabla \eta\rangle & \leq \frac{\varepsilon_5^{1+a}}{1+a}\int_s^t \langle |b_2|^{1+a} v^p \eta^2\rangle  +  \frac{a}{a+1}\varepsilon_5^{-\frac{a+1}{a}}\int_s^t \langle v^p |\nabla \eta|^{\frac{a+1}{a}}\rangle \\
& \leq \frac{\varepsilon_5^{1+a}}{1+a}\int_s^t \langle |b_2|^{1+a} v^p \eta^2\rangle  +  C_{\varepsilon_5}(R-r)^{-\frac{a+1}{a}}\int_s^t \langle v^p \eta^{\frac{a+1}{a}(1-\gamma)}\rangle.
\end{align*}
We apply \eqref{fb2} in the first term and $\eta^{\frac{a+1}{a}(1-\gamma)} \leq \mathbf{1}_{\{\eta>0\}}$ in the second term. Finally, assuming that $p$ is chosen so that $\frac{1+a}{p'} \geq 1$, we modify \eqref{h_term} as
\begin{align*}
\left|\int_s^t \langle |\mathsf{h}|f v^{p-1}\eta^2\rangle \right| \leq  \int_s^t \langle \mathbf{1}_{|\mathsf{h}| \geq 1}|\mathsf{h}|^{\frac{1+a}{p'}}|f| |v|^{p-1}\eta^2\rangle + \int_s^t \langle \mathbf{1}_{|\mathsf{h}| < 1}|\mathsf{h}||f| |v|^{p-1}\eta^2\rangle,
\end{align*}
so, after applying Young's inequality as in the proof, we can use condition \eqref{fb2} for $\mathsf{h}$.
Now we can repeat the rest of the proof of Proposition \ref{c_prop}. (We arrive at \eqref{c_ineq} with $(R-r)^{-\frac{a+1}{a}}\mathbf{1}_{\{\eta>0\}}$ instead of $|\nabla \eta|^2$, but this is what we need in Proposition \ref{mp} anyway.) Of course, the form-bound $\delta_2$ of $b_2$ can be arbitrarily large since we can choose $\varepsilon_5$ as small as needed.
\end{remark}

Recall: $\rho(x)=(1+\kappa|x|^{-2})^{-\beta}$, $\beta>\frac{d}{4}$, $\kappa>0$ is sufficiently small.

\begin{proposition}
\label{cor1}
Let $u$ be the solution to Cauchy problem \eqref{k_eq3}. Fix $T>0$ and $1<\theta<\frac{d}{d-1}$.
For all $p>p_\delta$, $p \geq 2$, there exists a generic constant $C$ such that 
\begin{align*}
\|u\|_{L^\infty([0,T] \times \mathbb R^d)} & \leq C 
\sup_{z \in \mathbb Z^d}\biggl(\int_0^T \bigg\langle \big(\mathbf{1}_{\{|\mathsf{h}| \geq 1\}} + \mathbf{1}_{\{|\mathsf{h}| < 1\}}|\mathsf{h}|^p\big)^{\theta'} |f|^{p\theta' }\rho^2_z\bigg\rangle\biggr)^{\frac{1}{p\theta'}}
\end{align*}
where $\rho_z(x):=\rho(x-z)$.
\end{proposition}

\begin{proof}[Proof of Proposition \ref{cor1}]
Applying $\rho \geq c_0\mathbf{1}_{B(0,1)}$ and \eqref{c_ineq2} to the last term in \eqref{mp_f} of Proposition \ref{mp}, we arrive at
\begin{align*}
\sup_{[0,T] \times B(0,\frac{1}{2})}u_+  & \leq C' 
\biggl(\int_0^T \big\langle \big(\mathbf{1}_{\{|\mathsf{h}| \geq 1\}} + \mathbf{1}_{\{|\mathsf{h}| < 1\}}|\mathsf{h}|^p\big)^{\theta'} |f|^{p\theta' }\rho^2\big\rangle\biggr)^{\frac{1}{p\theta'}} \\
& + C'' \biggl(\int_0^T \big\langle \big(\mathbf{1}_{\{|\mathsf{h}| \geq 1\}} + \mathbf{1}_{\{|\mathsf{h}| < 1\}}|\mathsf{h}|^p\big) |f|^{p}\rho^2\big\rangle\biggr)^{\frac{1}{p}} \\
& + C''' \biggl(\int_0^T \big\langle \big(\mathbf{1}_{\{|\mathsf{h}| \geq 1\}} + \mathbf{1}_{\{|\mathsf{h}| < 1\}}|\mathsf{h}|^p\big)^{\theta} |f|^{p\theta }\rho^2\big\rangle\biggr)^{\frac{1}{p\theta}} \equiv I_1+I_2+I_3.
\end{align*}
 Applying H\"{o}lder's inequality to $I_2$ and $I_3$ (using that $\theta'>\theta>1$), we arrive at 
$$
\|u\|_{L^\infty([0,T] \times B(0,\frac{1}{2}))}  \leq C 
\biggl(\int_0^T \bigg\langle \big(\mathbf{1}_{\{|\mathsf{h}| \geq 1\}} + \mathbf{1}_{\{|\mathsf{h}| < 1\}}|\mathsf{h}|^p\big)^{\theta'} |f|^{p\theta' }\rho^2\bigg\rangle\biggr)^{\frac{1}{p\theta'}}
$$
Since the choice of the centre of the ball $B(0,\frac{1}{2})$ was arbitrary, this ends the proof.
\end{proof}

\subsection{Proof of Theorem \ref{thm1}} \label{sde_subsect} Once Proposition \ref{cor1} is established, one can construct a martingale solution to \eqref{sde} via a standard tightness argument. The proof below, included for reader's convenience, follows \cite{ZZ, Zh,RZ}.

\begin{definition}
A probability measure $\mathbb P_x$ on the canonical space $\bigl(C([0,1],\mathbb R^d),\mathcal B_t=\sigma\{\omega_s \mid 0 \leq s \leq t\}\bigr)$ is called a martingale solution to the SDE \eqref{sde} if

1) $\mathbb P_x[\omega_0=x]=1$.

2)
$$
\mathbb E_{x}\int_0^t|b(s,\omega_s)|<\infty, \quad 0<t\leq 1 \qquad (\mathbb E_x:=\mathbb E_{\mathbb P_x}).
$$ 

3) For every $\varphi \in C_2^2$ the process
$$
M^\varphi_t:=\varphi(\omega_t)-\varphi(\omega_0) + \int_0^t (-\Delta \varphi + b \cdot \nabla \varphi)(s,\omega_s) ds
$$
is a martingale: $$\mathbb E_x[M^\varphi_{t_1} \mid \mathcal B_{t_0}]=M_{t_0}^\varphi$$ for all $0 \leq t_0<t_1 \leq 1$ $\mathbb P_x$-a.s. 
\end{definition}

Let $b$ be a vector field in $\mathbf{F}_\delta$, $\delta<4$, so in general it is locally unbounded. 
Let us fix bounded smooth vector fields $b_n \in C_c^\infty([0,\infty[ \times \mathbb R^d,\mathbb R^d) \cap \mathbf{F}_\delta$ (with $g=g_\delta$ independent of $n$) such that 
$$
b_n \rightarrow b \quad \text{ in } L^2_{\loc}([0,\infty[ \times \mathbb R^d,\mathbb R^d).
$$
Such vector fields can be constructed by multiplying $b$ by $\mathbf{1}_{\{0 \leq |t| \leq n, |x| \leq n, |b(x)| \leq n\}}$, which preserves the form-bound $\delta$, and then applying a K.\,Friedrichs mollifier in $(t,x)$, see \cite{KM} for details if needed. (In fact, we don't even need to include the indicator function, which allows to control the form-bound of ${\rm div}\,b$ \cite[Sect.\,3, 4]{KiS4}, cf.\,Remark \ref{div_rem}.)

Fix $x \in \mathbb R^d$. By a classical result, there exist strong solutions $X^n$ to the SDEs
$$
X^n_t=x-\int_0^t b_n(X^n_s)ds + \sqrt{2}dB_t, \quad n=1,2,\dots,
$$
where $B_t$ is a Brownian motion in $\mathbb R^d$ on a fixed complete probability space $(\Omega,\mathcal F,\mathcal F_t,\mathbf P)$.

Let $0 \leq t_0 < t_1 \leq 1$. Consider the terminal-value problem for $t \leq t_1$
$$
\partial_t u_n + \Delta u_n + b_n \cdot \nabla u_n + F=0, \quad u_n(t_1)=0,
$$
where $F \in C_c([0,1] \times \mathbb R^d)$.
Then the It\^{o} formula yields
$$
\mathbf E\int_{t_0}^{t_1} F(r,X^n_r)dr = u_n(t_0,X^n_{t_0}).
$$
Hence, selecting $F=|\mathsf{h}|f$, where $\mathsf{h} \in \mathbf{F}_\nu \cap C_c^\infty(\mathbb R^d,\mathbb R^d)$ and $f \in C_c$ are as in the previous section, we have by Proposition \ref{cor1}
\begin{equation*}
\left|\mathbf E\int_{t_0}^{t_1} |\mathsf{h}(s,X^n_s)|f(s,X^n_s)ds\right| \leq  
\sup_{z \in \mathbb Z^d}\biggl(\int_{t_0}^{t_1} \bigg\langle \big(\mathbf{1}_{\{|\mathsf{h}| \geq 1\}} + \mathbf{1}_{\{|\mathsf{h}| < 1\}}|\mathsf{h}|^p\big)^{\theta'} |f|^{p\theta' }\rho^2_z\bigg\rangle\biggr)^{\frac{1}{p\theta'}}.
\end{equation*}
Set $\mathbb P^n_x:=(\mathbf P \circ X^n)^{-1}$ -- probability measures on $\bigl(C([0,1],\mathbb R^d),\mathcal B_t\bigr)$. Then the last estimate can be rewritten as
\begin{equation}
\label{id}
\left|\mathbb E^n_x\int_{t_0}^{t_1} |\mathsf{h}(s,\omega_s)|f(s,\omega_s)ds\right| \leq  
\sup_{z \in \mathbb Z^d}\biggl(\int_{t_0}^{t_1} \bigg\langle \big(\mathbf{1}_{\{|\mathsf{h}| \geq 1\}} + \mathbf{1}_{\{|\mathsf{h}| < 1\}}|\mathsf{h}|^p\big)^{\theta'} |f|^{p\theta' }\rho^2_z\bigg\rangle\biggr)^{\frac{1}{p\theta'}},
\end{equation}
where $\mathbb E_x^n:=\mathbb E_{\mathbb P^n_x}$.
The following two instances of estimate \eqref{id} will yield the sought martingale solution:

\smallskip

1. \eqref{id} with $\mathsf{h}=b_n$ and $f \equiv 1$ (here $f \in C_c$ $\Rightarrow$ $f \equiv 1$  using Fatou's Lemma):
\begin{align*}
\left|\mathbb E^n_x\int_{t_0}^{t_1} |b_n(s,\omega_s)|ds\right| & \leq   
\sup_{z \in \mathbb Z^d}\biggl(\int_{t_0}^{t_1} \bigg\langle \big(\mathbf{1}_{\{|b_n| \geq 1\}} + \mathbf{1}_{\{|b_n| < 1\}}|b_n|^p\big)^{\theta'} \rho^2_z\bigg\rangle\biggr)^{\frac{1}{p\theta'}} \\
& \leq C (t_1-t_0)^\mu \quad \text{ for generic }\mu>0 \text{ and } C, 
\end{align*}
The latter allows to verify the tightness of $\{\mathbb P^n_x\}$, see \cite[proof of Theorem 1.1]{RZ}. Thus, there exists a subsequence $\{\mathbb P_x^{n_k}\}$ and a probability measure $\mathbb P_x$ on  $C([0,1],\mathbb R^d)$
such that
\begin{equation}
\label{w}
\mathbb P_x^{n_k} \rightarrow \mathbb P_x \text{ weakly }.
\end{equation}
Now, by \eqref{w} and the standard monotone class argument,
$$
\left|\mathbb E_x\int_{t_0}^{t_1} |b(s,\omega_s)|ds\right| \leq C (t_1-t_0)^\mu.
$$
Our goal now is to show that the limit measure $\mathbb P_x$ solves the martingale problem for \eqref{sde}.
It suffices to show that
$\mathbb E_x[M^\varphi_{t_1} G]=\mathbb E_x[M_{t_0}^\varphi G]$ for every $\mathcal B_{t_0}$-measurable $G \in C_b\big(C([0,T],\mathbb R^d)\big)$.
The task reduces to passing to the limit in $n$ in $\mathbb E^n_x[M^{\varphi,n}_{t_1} G]=\mathbb E^n_x[M_{t_0}^{\varphi,n}G],$ where 
$$
M^{\varphi,n}_t=\varphi(\omega_t)-\varphi(\omega_0) + \int_0^t (-\Delta\varphi + b_n \cdot \nabla \varphi)(s,\omega_s) ds.
$$
That is, we need to prove
\begin{equation}
\label{c}
\lim_{n_k}\mathbb E_x^{n_k}\int_0^t (b_{n_k} \cdot \nabla \varphi)(s,\omega_s) G(\omega)ds= \mathbb E_x\int_0^t (b \cdot \nabla \varphi)(s,\omega_s) G(\omega)ds
\end{equation}
This is done using the weak convergence \eqref{w} and  the next estimate.

\medskip

2. \eqref{id} with $\mathsf{h}:=b_{m_1}-b_{m_2} \in \mathbf{F}_{\sqrt{2}\delta}$, $f:=|\nabla \varphi|$:
\begin{align*}
& \left|\mathbb E^n_x\int_{t_0}^{t_1} \big|b_{m_1}(s,\omega_s)-b_{m_2}(s,\omega_s)\big||\nabla \varphi(s,\omega_s)|ds\right| \\
&\leq   
\sup_{z \in \mathbb Z^d}\biggl(\int_{t_0}^{t_1} \bigg\langle \big(\mathbf{1}_{\{|b_{m_1}-b_{m_2}| \geq 1\}} + \mathbf{1}_{\{|b_{m_1}-b_{m_2}| < 1\}}|b_{m_1}-b_{m_2}|^p|\nabla \varphi|^{p\theta'}\big)^{\theta'} \rho^2_z\bigg\rangle\biggr)^{\frac{1}{p\theta'}}.
\end{align*}
Without loss of generality, $|b-b_{n_k}| \rightarrow 0$ a.e.
Since $\varphi$ has compact support, the RHS converges to $0$ as $m_1$, $m_2 \rightarrow \infty$.
It follows from the weak convergence \eqref{w} and the standard monotone class argument that  
\begin{align*}
& \left|\mathbb E_x\int_{t_0}^{t_1} \big|b(s,\omega_s)-b_{m}(s,\omega_s)\big||\nabla \varphi(s,\omega_s)|ds\right| \\
&\leq   
\sup_{z \in \mathbb Z^d}\biggl(\int_{t_0}^{t_1} \bigg\langle \big(\mathbf{1}_{\{|b-b_{m}| \geq 1\}} + \mathbf{1}_{\{|b-b_{m}| < 1\}}|b-b_{m}|^p|\nabla \varphi|^{p\theta'}\big)^{\theta'} \rho^2_z\bigg\rangle\biggr)^{\frac{1}{p\theta'}},
\end{align*}
where the RHS converges to $0$ as $m \rightarrow \infty$. Now, we prove \eqref{c}:
\begin{align*}
&\left| \mathbb E_x^{n_k}\int_0^t (b_{n_k} \cdot \nabla \varphi)(s,\omega_s) G(\omega)ds - \mathbb E_x\int_0^t (b \cdot \nabla \varphi)(s,\omega_s) G(\omega)ds\right| \\
& \leq \left| \mathbb E_x^{n_k}\int_0^t |b_{n_k}-b_m| |\nabla \varphi|(s,\omega_s) |G(\omega)|ds \right| \\
& + \left| \mathbb E_x^{n_k}\int_0^t (b_{m} \cdot \nabla \varphi)(s,\omega_s) G(\omega)ds  - \mathbb E_x\int_0^t (b_{m} \cdot \nabla \varphi)(s,\omega_s) G(\omega)ds\right| \\
& + \left| \mathbb E_x\int_0^t |b_{m}-b| |\nabla \varphi|(s,\omega_s) |G(\omega)|ds \right|,
\end{align*}
where the first and the third terms in the RHS can be made arbitrarily small using the estimates above and the boundedness of $G$ by selecting $m$, and then $n_k$, sufficiently large. The second term can be made arbitrarily small in view of \eqref{w} by selecting $n_k$ even larger.
This ends the proof of Theorem \ref{thm1}.

\begin{remark}
Let $b \in \mathbf{F}_\delta$, $\delta<4$. Let $u_n$ be defined by
\begin{equation*}
(\partial_t - \Delta + b_n \cdot \nabla) u_n=0, \qquad u_n(0)=g \in C_b \cap L^1,
\end{equation*}
where $b_n$ are as above.

1.~For every $p>p_\delta$, the limit
\begin{equation}
\label{k_conv}
u:=s{\mbox-}L^p{\mbox-}\lim_{n} u_n  \quad \text{loc.\,uniformly in $t \geq 0$},
\end{equation}
exists and determines a unique weak solution (in $L^p$) to Cauchy problem
$(\partial_t-\Delta + b \cdot \nabla) u=0$, $u(0+)=g$. See \cite{S}.

2.~One can apply Moser's method in $L^p$, $p>p_\delta$, $p \geq 2$ to show H\"{o}lder continuity of the weak solution $u$. Combined with \eqref{k_conv}, this allows to conclude that $u_n \rightarrow u$ everywhere on $\mathbb R^d$. We plan to address these matters in detail elsewhere.
\end{remark}

\end{document}